\documentclass[12pt,reqno]{amsart}
\usepackage{amssymb}
\usepackage{graphicx}
\usepackage[usenames,dvipsnames]{color}
\usepackage{euscript}
\usepackage{multicol}
\usepackage{amssymb}
\usepackage{amsmath}
\usepackage{times} 
\usepackage{colordvi}

\newtheorem{theorem}{Theorem}[section]
\newtheorem{corollary}[theorem]{Corollary}
\newtheorem{lemma}[theorem]{Lemma}

\newtheorem{remark}[theorem]{Remark}

\numberwithin{equation}{section}

\begin{document}

\title[Hardy inequalities with homogenuous weights]
{Hardy inequalities with homogenuous weights}

\author{Thomas Hoffmann-Ostenhof}
\address{Thomas Hoffmann-Ostenhof:
University of Vienna}
\email{thoffmann@tbi.univie.ac.at}

\author{Ari Laptev}
\address{Ari Laptev: Imperial College London }
\email{a.laptev@imperial.ac.uk}



\begin{abstract}
In this paper we obtain some sharp Hardy inequalities with weight functions that may admit singularities on the unit sphere.
 In order to prove the main results of the paper we use some recent sharp inequalities for the lowest eigenvalue
 of Schr\"odinger operators on the unit sphere obtaind in the paper \cite{DEL}.
\end{abstract}

\maketitle

\section{Introduction}

\noindent
The classical Hardy inequality for the Laplacian in $\mathbb R^d$ is 
\begin{equation}\label{clHardy}
\int_{\mathbb R^d} |\nabla u(x)|^2\, dx \ge \frac{(d-2)^2}{4}\, \int_{\mathbb R^d} \frac{|u(x)|^2}{|x|^2}\, dx, \quad u\in C_0^\infty(\mathbb R^d),
 \quad d\ge3, 
\end{equation}
is well known and has many elementary proofs. This inequality is not  achieved but the constant $(d-2)^2/4$ is sharp. 
It is often assiciated with the Heisenberg uncertainty principle and plays important role in spectral theory of Schr\"odinger 
operators. In particular, this inequality is equivalent to the quadratic form inequality
$$
-\Delta -  \frac{(d-2)^2}{4}\, \frac{1}{|x|^2} \ge 0,
$$
which states that if $d\ge3$, then one can subtract a positive operator from the Laplacian so that the difference remains 
non-negative.The literature devoted to different types of Hardy's inequalities is vast and it is not our aim to cover 
it in this short paper, but note that description of other ``Hardy weights" is an interesting problem.  
Here we are dealing with the case, where instead of the spherical symmetrical weight $1/|x|^2$ in the integral in the 
right hand side of \eqref{clHardy} we consider a more general class of homogeneous functions of degree $-2$ which may 
have singularities along  rays starting at the origin.

\medskip
\noindent
Namely, in this paper we prove the inequality 
\begin{equation}\label{newHardy1}
\int_{\mathbb R^d} |\nabla u(x)|^2\, dx \ge \tau \int_{\mathbb R^d} \frac{\Phi(x/|x|) }{|x|^2}\, |u(x)|^2\, dx,
 \quad u\in C_0^\infty(\mathbb R^d), \quad d\ge 3,
\end{equation}
with some $\tau>0$ for a class of measurable functions $\Phi$ defined  on $\mathbb S^{d-1}$.  The theorems proved in this 
paper are based on the recent inequalities obtained in the paper \cite{DEL}, where the authors have found sharp bounds 
for the first eigenvalue of a Schr\"odinger operator on $\mathbb S^{d-1}$ using deep results from \cite{BV}.  

\medskip
\noindent
In order to formulate our results let us introduce the measure $d\vartheta$ induced by Lebesgue's measure on 
$\mathbb S^{d-1}\subset \mathbb R^d$. We denote by $\|\cdot\|_{L^p(\mathbb S^{d-1})}$ the quantity 
$$
\|\Phi\|_{L^p(\mathbb S^{d-1})} = \left(\int_{\mathbb S^{d-1}} |\Phi(\vartheta)|^p\, d\vartheta\right)^{1/p}.
$$
Our first result is: 
\begin{theorem}\label{main} 
Let $d\ge3$ and $0\le \Phi \in L^{p}(\mathbb S^{d-1})$,  where
\begin{equation}\label{cond_p}
p\ge \frac{(d-2)^2}{2(d-1)} + 1.
\end{equation}
Then  
\begin{equation}\label{newHardy2}
\int_{\mathbb R^d} |\nabla u(x)|^2\, dx \ge \tau \int_{\mathbb R^d} \frac{\Phi(x/|x|) }{|x|^2}\, |u(x)|^2\, dx, \qquad u\in C_0^\infty(\mathbb R^d),
\end{equation}
where 
\begin{equation}\label{sharp_tau}
\tau =    \frac{(d-2)^2}{4}\, |\mathbb S^{d-1}|^{1/p}  \,   \| \Phi\|_{L^p(\mathbb S^{d-1})}^{-1}.
\end{equation}
\end{theorem} 

\begin{remark} 
For the class of functions $\Phi$ satisfying the conditions of the theorem, inequality \eqref{newHardy2} is sharp. Indeed, if $\Phi\equiv 1$, then \eqref{newHardy2} takes the classical sharp form
$$
\int_{\mathbb R^d} |\nabla u(x)|^2\, dx \ge  \frac{(d-2)^2}{4}\, \int_{\mathbb R^d} \frac{|u(x)|^2}{|x|^2}\, dx.
$$
\end{remark}

\begin{remark} 
If for example $d=3$, then the lowest possible value of $p$ that is allowed in Theorem \ref{main} equals $5/4$, see \eqref{cond_p}.
\end{remark}


\medskip
\noindent
Note that the condition on the value of $p$ in \eqref{cond_p} could be weakened. In our next theorem we consider the values of 
$p$ smaller than $ \frac{(d-2)^2}{2(d-1)} + 1$.


\begin{theorem}\label{Main2}
Let $d\ge3$ and $0\le \Phi\in L^{p}(\mathbb S^{d-1})$, where 
$$
p\in \left(1, \, 5/4\right),\,\,{\rm if}\,d=3, \quad{\rm and} \quad 
p\in \left[ \frac{d-1}{2}, \,\frac{(d-2)^2}{2(d-1)} + 1\right), \, \,  {\rm if} \, \,  d\ge4.
$$
Then  
\begin{equation}\label{newHardy4}
\int_{\mathbb R^d} |\nabla u|^2\, dx  \\
\ge (1-\nu_0) \, \frac{(d-2)^2}{4} \, \int_{\mathbb R^d} \frac{|u|^2}{|x|^2} \, dx + \tau\, \int_{\mathbb R^d} \frac{\Phi(x/|x|)}{|x|^2} \, |u|^2\, dx,
\end{equation}
where 
$$
\nu_0 =  \frac{2 (d-1)(p-1)}{(d-2)^2} <1.
$$
and
\begin{equation}\label{tau}
\tau =   \nu_0\, \frac{(d-2)^2}{4}\,  |\mathbb S^{d-1}|^{1/p}  \,   \| \Phi\|_{L^p(\mathbb S^{d-1})}^{-1}.
\end{equation}
\end{theorem}
\begin{remark}
The inequality \eqref{newHardy4} is sharp and achieved for the functions $\Phi \equiv const$. Moreover, if 
$p= \frac{(d-2)^2}{2(d-1)} + 1$, then $\nu_0 = 1$ in  \eqref{newHardy4} and this inequality coincides with  \eqref{newHardy2}. 
\end{remark}


\begin{remark}\label{remark2}
In Theorem  \ref{theorem2} (see section 4)  we consider the values of $p$ 
\begin{equation}\label{p}
 \frac{d-1}{2}  < p   < \frac{(d-2)^2}{2(d-1)} + 1
\end{equation}
and obtain an inequality similar to \eqref{newHardy4} with with different ranges of $\tau $ and $\nu$'s.  
It is interesting that in this case the optimal  class of functions $\Phi$ does not coincide with constants. 
It is more convenient for us to formulate and prove the respective result after the proof of Theorems \ref{main} and \ref{Main2}.
\end{remark}

\medskip
\noindent
Finally in the last section we obtain a Hardy inequality for fractional powers of the Laplacian. Namely, let us  define the quadratic form 
$$
\int_{\Bbb R^d} |\nabla^\varkappa u(x)|^2\, dx = (2\pi)^{-d}\, \int_{\Bbb R^d} |\xi|^{2\varkappa} |\hat u(\xi)|^2 \, d\xi,
$$
where $\hat u$ is the Fourier transform of $u$. 

 \begin{theorem}\label{FracLaplacians}
Let $0<\varkappa < d/2$ for $d=1,2$, and  $0<\varkappa\le 1$ for $d\ge3$. Assume that  $\Phi=\Phi(x/|x|)\ge 0$ is a measurable function defined on $\Bbb S^{d-1}$, such that $\Phi\in L^{d/2\varkappa} (\mathbb S^{d-1})$. 
Then 
\begin{equation}\label{P-Sz-hardy}
\int_{\mathbb R^d}|\nabla^\varkappa (x)|^2\ge \tau \int_{\mathbb
R^d}\frac{\Phi(x/|x|)}{|x|^2 }\, |u(x)|^2 \, dx,
\end{equation}
where 
\begin{equation}\label{tau-frac}
\tau = 2^{2\varkappa} \, \frac{\Gamma^2\left((d/2+\varkappa)/2\right)}{\Gamma^2\left((d/2-\varkappa)/2\right)}\, 
\left|\mathbb S^{d-1}\right|^{2\varkappa/d}\, \|\Phi\|_{L^{d/2\varkappa}(\mathbb S^{d-1})}^{-1}.
\end{equation}
\end{theorem}

\bigskip
\noindent
In order to prove this theorem we use fractional Hardy inequalities proved in \cite{H} and \cite{Ya}Ê   \Big(note that 
$2^{2\varkappa} \, \Gamma^2\left((d/2+\varkappa)/2\right)\, \Gamma^{-2}\left((d/2-\varkappa)/2\right)\big|_{\varkappa=1} =(d-2)^2/4$ \Big).

\begin{remark}
Note, that in the case $\varkappa=1$ Theorem \ref{main} is stronger than Theorem \ref{FracLaplacians} since it allows us to have a larger class of functions  $\Phi$ because of the strict embedding 
$$
 L^{d/2}(\mathbb S^{d-1})  \subset L^{\frac{(d-2)^2}{2(d-1)} + 1}(\mathbb S^{d-1}).
$$
\end{remark} 

\begin{remark}
The constant $\tau$ in \eqref{tau-frac} is sharp as it is sharp for $\Phi = {\rm const}$.
\end{remark}

\noindent
In the recent paper of B. Devyver, M. Fraas and Y. Pinchover \cite{DFP} the authors considered a rather general second order operator with variable coefficients and found an optimal weight for the respective Hardy inequality. In particular, such a weight for the Laplacian coincides with  $1/|x|^2$.

\noindent
Our result is different as we deal with the ``flat" Laplacian and find a class of weight functions that may have singularities not only at the origin.

\medskip
\noindent
{\it Acknowledgements.}
The authors express their gratitude to Rupert Frank and Michael Loss for valuable discussions.

\section{Auxiliary statements}

\noindent 
In order to prove Theorem \ref{main} we use a result obtained in \cite {DEL} which provides a sharp estimate for the first 
negative eigenvalue $\lambda_1$ of the Schr\"odinger operator  in $L^2(\mathbb S^{d-1})$,
$$
-\Delta_\vartheta - \Phi, \qquad \Phi\ge 0,
$$
where $-\Delta_\vartheta$ is the Laplace-Beltrami operator on $\mathbb S^{d-1}$. Note that we need it only for the case $d\ge3$.

\begin{theorem}\label{AxTheorem} 
Let $d\ge3$ and  $0\le \Phi \in L^{p} (\mathbb S^{d-1})$,  where $p\in\big((d-1)/2,+\infty\big)$. Then there exists 
an increasing function $\alpha:\mathbb R_+\to\mathbb R_+$  
\begin{equation}\label{linear} 
\alpha(\mu)=\mu \quad {\rm for\,  any} \quad  \mu\in\left[0,\frac{d-1}{2}\,(p-1)\right],
\end{equation} 
and convex if  $\mu\in\big(\frac{d-1}{2}\,(p-1),+\infty\big)$, such that
\begin{equation}\label{T11}
|\lambda_1(-\Delta_\vartheta -\Phi)| \le \alpha\left(\frac{1}{|\mathbb S^{d-1}|^{1/p} } \|\Phi\|_{L^p(\mathbb S^{d-1})}\right).
\end{equation}
The estimate \eqref{T11} is optimal in the sense that there exists a non-negative function $\Phi$, such that 
\begin{equation*}
|\lambda_1(-\Delta_\vartheta -\Phi)| = \alpha\left(\frac{1}{|\mathbb S^{d-1}|^{1/p} } \|\Phi\|_{L^p(\mathbb S^{d-1})}\right).
\end{equation*}
for any $\mu\in\big(\frac{d-1}{2}\,(p-1),+\infty\big)$. If $\mu\le \frac{d-1}{2}\,(p-1)$, then equality in \eqref{T11}
is achieved for constants.

\noindent
For large values of $\mu$ we have
\begin{equation}\label{large-norms}
\alpha(\mu)^{p-\frac{d-1}{2}}= L^1_{p-\frac{d-1}{2}, d-1}  \, \mu^p\,(1+o(1)),
\end{equation}
where $L^1_{\gamma,d-1}$ are the Lieb-Thirring constants appearing in \cite{LTh} in the inequality for the lowest
 eigenvalue of a Schr\"odinger operator in $L^2(\Bbb R^{d-1})$. 

\noindent
Moreover, if $p=(d-1)/2$, $d\ge4$, then \eqref{T11} is satisfied with $\alpha(\mu) = \mu$ for $\mu \in [0,(d-1)(d-3)/2]$. 
\end{theorem}

\medskip
\noindent
Note that here the function $\alpha(\mu)$ is invertible and its inverse $\mu(\alpha)$ equals (see \cite{DEL})
\begin{equation}\label{mu-alpha}
\mu(\alpha) = |\mathbb S^{d-1}|^{\frac{2}{q} -1}\, \inf_{u\in H^1(S^{d-1})}\, \frac{\|\nabla u\|^2_{L^2(\Bbb S^{d-1})} + \alpha \, 
\| u\|^2_{L^2(\Bbb S^{d-1})}}{\|u\|^2_{L^q(\Bbb S^{d-1})}},
\end{equation}
where $q\in\left(2, \frac{2(d-1)}{d-3}\right)$ (with  $(2,\infty)$ for $d=3$).
The optimal value in \eqref{mu-alpha} is achieved by the unique solution $u$ of the non-linear equation
$$
-\Delta u + \alpha \, u - \mu(\alpha)\, u^{q-1} = 0,
$$
that for each chosen $\alpha$ also defines the value of $\mu(\alpha)$.

\noindent
Obviously if  $v\equiv c$, $c\in\mathbb R$, and $\Phi\ge 0$ is non-trivial,  then  the quadratic form 
$$
\int_{\mathbb S^{d-1}} \left(|\nabla_\vartheta v|^2 - \Phi |v|^2\right)\, d\vartheta  = 
-c^2\, \int_{\mathbb S^{d-1}} \Phi \, d\vartheta <0.
$$
Therefore due to the variational principle the eigenvalue $\lambda_1(-\Delta_\vartheta - \Phi)$ is negative
 for any nonnegative, non-trivial $\Phi$ and consequently the inequality~\eqref{T11} is a lower estimate
\begin{equation} \label{alphamu}
0\ge\lambda_1(-\Delta_\vartheta - \Phi)\ge- \alpha\left(\frac{1}{|\mathbb S^{d-1}|^{1/p} } 
\|\Phi\|_{L^p(\mathbb S^{d-1})}\right) \quad \forall\, \Phi \in L^p(\mathbb S^{d-1}).
\end{equation} 
If $\Phi$ changes sign, the above inequality still holds if $\Phi$ is replaced by the positive part $\Phi_+$ 
of $\Phi$, provided the lowest eigenvalue is negative. We can then write
\begin{equation*}
|\lambda_1(-\Delta_\vartheta - \Phi )|\le      \alpha\left(\frac{1}{|\mathbb S^{d-1}|^{1/p} } \|\Phi_+
\|_{L^p(\mathbb S^{d-1})}\right).
\end{equation*} 

\medskip
\noindent
The expressions for the constants  $ L^1_{p-\frac{d-1}{2}, d}$ in \eqref{large-norms} are not explicit for 
$d\ge3$, but can be given in terms of an optimal constant in some Gagliardo-Nirenberg-Sobolev inequality 
(see~\cite{LTh} and \cite{DEL}) in the following way:

\noindent
Let $q= 2p/(p-1)>2$ and denote by $\mathsf K_{\rm GN}(q,d-1)$ the optimal constant in the Gagliardo-Nirenberg-Sobolev 
inequality, given by
\begin{equation*}
K_{\rm GN}(q,d-1):=\inf_{u\in H^1(\Bbb R^{d-1})\setminus\{0\}}\, \frac{\|\nabla u\|^{2\,\rho}_{L^2(\Bbb R^{d-1})} \,  \|u\|^{2\,(1-\,\rho)}_{L^2(\Bbb R^{d-1})}}{\|u\|^2_{L^q(\Bbb R^{d-1})}},
\end{equation*}
where  $\rho=\rho(q,d)=(d-1)\,\frac{q-2}{2\,q}$.

\noindent
Then  
\begin{equation*}
 L^1_{p-\frac{d-1}{2}, d-1}=\left[\rho^{-\rho}\,(1-\,\rho)^{-\,(1-\,\rho)}\,\mathsf K_{\rm GN}(q,d-1)\right]^{-p} \,.
\end{equation*}

\medskip
\noindent
\begin{lemma}\label{lemma}

Let $\tau>0$ and  $d\ge3$. Then  
\begin{multline}\label{eq-lem}
\int_{\mathbb R^d}|\nabla u|^2 dx \\
\ge \int_{\mathbb 
R^d}\frac{|u|^2}{|x|^2}\left(\tau\, \Phi(x/|x|) +\lambda_1\left(-\Delta_\vartheta - \tau \, \Phi(x/|x|)\right)
+\frac{(d-2)^2}{4}\right) \, dx.
\end{multline}
\end{lemma}


\begin{proof}
Let $x= (r, \vartheta) \in \mathbb R^d$ be polar coordinates in $ \mathbb R^d$. Then  we find
\begin{equation}\label{1}
\int_{\mathbb R^d} |\nabla u|^2\, dx = \int_0^\infty\int_{\mathbb S^{d-1}} \left(|\partial_r u|^2 + 
\frac{1}{r^2} \, |\nabla_\vartheta u|^2\right)\, r^{d-1}\, d\vartheta dr.
\end{equation} 
Note that according to the classical Hardy inequality for radial functions $f\in C_0^\infty(0,\infty)$ we have 
$$
\int_0^\infty |f'(r)|^2 \, r^{d-1}\, dr \ge  \frac{(d-2)^2}{4}\, \int_0^\infty \frac{|f|^2}{r^2}\, r^{d-1}\, dr.
$$
Applying the latter inequality to $u(r,\vartheta)$ for a fixed $\vartheta$ and then integrating over $\mathbb S^{d-1}$ we obtain
\begin{equation}\label{2}
\int_{\mathbb S^{d-1}}  \int_0^\infty  |\partial_r u|^2 \, r^{d-1}\, dr d\vartheta\\
\ge \frac{(d-2)^2}{4}\,  \int_{\mathbb S^{d-1}}  \int_0^\infty  \frac{|u|^2}{r^2} \, r^{d-1}\, dr d\vartheta.
\end{equation} 
Let $\tau>0$. It follows from Theorem \ref{AxTheorem}  that 
\begin{multline}\label{3}
\int_0^\infty\int_{\mathbb S^{d-1}} \frac{1}{r^2} \, |\nabla_\vartheta u|^2 \, r^{d-1}\, d\vartheta dr 
= \int_0^\infty\int_{\mathbb S^{d-1}} \frac{1}{r^2} \,  |\nabla_\vartheta u|^2 \, r^{d-1}\, d\vartheta dr 
\\
= 
\int_0^\infty\int_{\mathbb S^{d-1}} \frac{1}{r^2} \, \tau \,  \Phi \,  |u|^2 \, \, r^{d-1}\, d\vartheta dr  + 
\int_0^\infty\int_{\mathbb S^{d-1}} \frac{1}{r^2}  
\left( |\nabla_\vartheta u|^2  - \tau \,  \Phi \,  |u|^2\right)\, \, r^{d-1}\,
d\vartheta dr \\
\ge 
\int_0^\infty\int_{\mathbb S^{d-1}} \frac{1}{r^2} \, \left(\tau \,  \Phi + \lambda_1(-\Delta_\vartheta -\tau \, \Phi) \right)\,|u|^2 \, r^{d-1}\, d\vartheta dr.
\end{multline} 
Putting together  \eqref{1}, \eqref{2} and \eqref{3} we obtain the statement of the lemma.
\end{proof} 
\medskip


\begin{corollary}
Let $\tau>0$ and  $d\ge3$ and  let $0\le \Phi \in L^{p} (\mathbb S^{d-1})$, where 
$$
p\in\big(\max\{1,(d-1)/2\},+\infty\big).
$$
Then
\begin{equation}\label{cor}
\int_{\mathbb R^d}|\nabla u|^2 dx 
\ge \int_{\mathbb 
R^d}\frac{|u|^2}{|x|^2}\left(\tau\, \Phi(x/|x|) - \alpha(\mu)
+\frac{(d-2)^2}{4} \right) \, dx,
\end{equation}
where 
$$
\mu = \tau \, \, |\mathbb S^{d-1}|^{-1/p}  \,  \|\Phi\|_{L^p(\mathbb S^{d-1})}.
$$
\end{corollary}

\begin{proof}
Indeed, in  order to prove \eqref{cor} it is enough to apply the inequality \eqref{alphamu} estimating the value of  
$\lambda_1\left(-\Delta_\vartheta - \tau \, \Phi(x/|x|)\right)$ in \eqref{eq-lem}
\end{proof} 


\medskip
\section{Proofs of the main results}

\noindent
{\it Proof of Theorem \ref{main}}.

\medskip
\noindent
The condition 
\begin{equation}\label{restr_p}
p\ge \frac{(d-2)^2}{2(d-1)} + 1 
\end{equation}
implies both 
$$
p\in \left(\frac{d-1}{2}, \infty\right) \quad {\rm and} \quad \frac{(d-2)^2}{4} \le \frac{d-1}{2} (p-1).
$$
Due to Theorem \ref{AxTheorem} the convex function $\alpha(\mu) = \mu$ for  
$$
\mu\in \left[0, \frac{(d-1)(p-1)}{2}\right]. 
$$
Thus if in \eqref{cor} we choose $\tau$ according to the equation
$$
\alpha(\mu) = \mu =  |\mathbb S^{d-1}|^{-1/p}  \,  \tau\, \| \Phi\|_{L^p(\mathbb S^{d-1})} =  \frac{(d-2)^2}{4},
$$
namely 
$$
\tau =   \frac{(d-2)^2}{4}\, |\mathbb S^{d-1}|^{1/p}  \,   \| \Phi\|_{L^p(\mathbb S^{d-1})}^{-1},
$$
then we obtain the statement of Theorem \ref{main}.


\medskip
\noindent
{\it Proof of Theorem \ref{Main2}}.

\medskip
\noindent

\medskip
\noindent
When proving Theorem \ref{main} we fully compensated the positive term in the right hand side of \eqref{2}. This gave us
 a restriction on the possible values of $p$, see \eqref{restr_p}. Assume now that 
\begin{equation}\label{p-new}
p\in \left(1, \, 5/4\right),\,\,{\rm if}\,d=3, \quad{\rm and} \quad 
p\in \left[ \frac{d-1}{2}, \,\frac{(d-2)^2}{2(d-1)} + 1\right), \, \,  {\rm if} \, \,  d\ge4,
\end{equation}
and choose $\nu_0$ such that 
\begin{equation}\label{lambda}
\nu_0\, \frac{(d-2)^2}{4} = \frac{(d-1)(p-1)}{2},
\end{equation}
which gives us the value 
$$
\nu_0 =  \frac{2 (d-1)(p-1)}{(d-2)^2} <1.
$$
Then using \eqref{cor} we find

\begin{multline}
\int_{\mathbb R^d}|\nabla u|^2 dx 
\ge \int_{\mathbb 
R^d}\frac{|u|^2}{|x|^2}\left(\tau\, \Phi(x/|x|) - \alpha(\mu)
+\frac{(d-2)^2}{4}\right) \, dx \\
= 
\int_{\mathbb 
R^d}  \left(\tau\, \Phi(x/|x|)  + (1-\nu_0)\, \frac{(d-2)^2}{4}\right)  \,  \frac{|u|^2}{|x|^2} \, dx \\
+ 
\int_{\mathbb R^d} 
  \left(  \nu_0\, \frac{(d-2)^2}{4}  - \alpha(\mu) \right) \, \frac{|u|^2}{|x|^2} \, dx.
\end{multline}
Due to the choice of $p$ and $\nu_0$ given in  \eqref{p-new} and \eqref{lambda} respectively, we have 
$$
\alpha(\mu) = \mu =  |\mathbb S^{d-1}|^{-1/p}  \,  \tau\, \| \Phi\|_{L^p(\mathbb S^{d-1})}.
$$ 
It remains to choose $\tau$ according to 
$$
 \tau\, \, |\mathbb S^{d-1}|^{-1/p}  \,  \| \Phi\|_{L^p(\mathbb S^{d-1})} =  \nu_0\, \frac{(d-2)^2}{4},
$$
namely, 
$$
\tau =   \nu_0\,  \frac{(d-2)^2}{4}\, |\mathbb S^{d-1}|^{1/p}  \,   \| \Phi\|_{L^p(\mathbb S^{d-1})}^{-1}.
$$
This completes the proof of Theorem \ref{Main2}.


\medskip
\section{Hardy inequalities with $\nu_0<\nu\le1$. }

\medskip
\noindent
As it was mentioned in Remark \ref{remark2}, for the values 
$$
\frac{d-1}{2}  < p   < \frac{(d-2)^2}{2(d-1)} + 1.
$$
we can now consider $\nu: \, \nu_0<\nu\le1$. 
Then  since
$$
\frac{(d-2)^2}{4}  > \frac{(d-1)(p-1)}{2} 
$$
 the equation 
$$
\alpha \left(\tau\,  \, |\mathbb S^{d-1}|^{-1/p}  \,  \| \Phi\|_{L^p(\mathbb S^{d-1})} \right) = \nu \, \frac{(d-2)^2}{4}
$$
is more complicated, because in this case $\alpha(\mu)$ is non-linear. 
However, since it is increasing and convex, its inverse $\mu(\alpha)$ is well defined and thus we find 
$$
\tau = |S^{d-1}|^{1/p}  \, \| \Phi\|_{L^p(\mathbb S^{d-1})}^{-1} \, \mu\left(\nu\, \, \frac{(d-2)^2}{4}\right).
$$
Hence the inequality \eqref{cor} immediately implies:

\begin{theorem}\label{theorem2}
Let $d\ge3$ and $0\le \Phi\in L^{p}(\mathbb S^{d-1})$, where 
\begin{equation*}
\frac{d-1}{2} < p   < \frac{(d-2)^2}{2(d-1)} + 1.
\end{equation*}
Then  
\begin{equation}\label{newHardy3}
\int_{\mathbb R^d} |\nabla u|^2\, dx  \\
\ge (1-\nu) \, \frac{(d-2)^2}{4} \, \int_{\mathbb R^d} \frac{|u|^2}{|x|^2} \, dx + \tau\, \int_{\mathbb R^d} \frac{\Phi(x/|x|)}{|x|^2} \, |u|^2\, dx,
\end{equation}
where 
$$
\nu_0 =  \frac{2 (d-1)(p-1)}{(d-2)^2}  < \nu \le 1
$$ 
and
$$
\tau =   |S^{d-1}|^{1/p}  \,  \| \Phi\|_{L^p(\mathbb S^{d-1})}^{-1} \,\,  \mu\left(\nu\, \frac{(d-2)^2}{4}\right).
$$
\end{theorem} 

.
\begin{remark}
Note that since $\mu(\alpha)$ is an increasing function, the value of $\tau$ in \eqref{newHardy3} is larger than the respecive value of $\tau$ in 
\eqref{newHardy4}.  In particular,  $\nu= 1$ allows us to consider a  class of weight functions $\Phi$ with full compensation of the term $(d-2)^2/4$. If follows from \cite{DEL} that the optimal functions $\Phi$ are not constants.
\end{remark}

\begin{remark}
The equation \eqref{large-norms} immediately implies 
$$
\mu(\alpha) = \left(L^1_{p-\frac{d-1}{2}, d-1}\right)^{-1/p} \, \alpha^{1-\frac{d-1}{2p}} \left(1+ o(1)\right) \quad 
{\rm as} \quad \alpha\to\infty,
$$
(see also Proposition 10 \cite{DEL}). 
\end{remark}


\medskip
\section{Proof of Theorem \ref{FracLaplacians}}

\medskip
\noindent
Let $A\subset\mathbb R^d$Ê  and denote by $A^*= \{x:\, |x|<r\}$ with $(|\mathbb S^{d-1}|/d ) |x|^d = |A|$ that is the symmetric rearrangement of $A$. By $\chi_A$ and $\chi_{A^*}$ we denote characteristic functions of $A$ and 
$A^*$ respectively. Then for any Borel measurable function $f:\, \mathbb R^d \to \mathbb C$  vanishing at infinity we denote by $f^*$ its decreasing rearrangement
$$
f^*(x) = \int_0^\infty \chi_{\{|f(x)| >t\}^*} \, dt.
$$
By using the Hardy-Littlewood rearrangement inequality we find 
\begin{equation*}
\int_{\mathbb R^d} \frac{\Phi(x/|x|)}{|x|^{2\varkappa}} \, |u|^2\, dx \le 
\int_{\mathbb R^d} \left(\frac{\Phi(x/|x|)}{|x|^{2\varkappa}}\right)^* \, (u^*)^2\, dx.
\end{equation*} 
Clearly
$$
\left|\{x: \, |\Phi(x/|x|)| >t\, |x|^{2\varkappa}\}\right| = \frac{1}{d}\, \,  t^{-d/2\varkappa}\,  \int_{\mathbb S^{d-1}}Ê\Phi^{d/2\varkappa} (\theta)\, d\theta
$$
and thus
\begin{multline*}
\left(\frac{\Phi(x/|x|)}{|x|^{2\varkappa}}\right)^* = \int_0^\infty \chi_{\{|\Phi(x/|x|)| >t\, |x|^{2\varkappa}\}^*} \, dt \\
= 
\int_0^\infty \chi_{\left\{|\mathbb S^{d-1}|\, |x|^d < \int_{\mathbb S^{d-1}} \Phi^{d/2\varkappa}(\theta)\, 
d\theta \, t^{-d/2\varkappa}\right\}}\, dt\\
= 
\frac{1}{\left|\mathbb S^{d-1}\right|^{2\varkappa/d}} \, \frac{\left(\int_{\mathbb S^{d-1}} \Phi^{d/2\varkappa}(\theta)\,  
d\theta\right)^{2\varkappa/d}}{|x|^{2\varkappa}}.
\end{multline*}Ê
We now use the Hardy inequality obtained in the papers \cite{H}, \cite{Ya} (see also \cite{FS}
for $L^p$-versions of these inequalities) stating that if $\varkappa < d/2$, then 
$$
\int_{\mathbb R^d} \frac{|u|^2}{|x|^{2\varkappa}} \, dx \le C_{\varkappa} \, \int_{\mathbb R^d} |\nabla^{\varkappa} u|^2\, dx,
$$
where
$$
C_\varkappa = 2^{-2\varkappa} \, \frac{\Gamma^2\left((d/2-\varkappa)/2\right)}{\Gamma^2\left((d/2+\varkappa)/2\right)}.
$$
Therefore 
\begin{multline*}\label{p=d/2kappa} 
\int_{\mathbb R^d} \frac{\Phi(x/|x|)}{|x|^{2\varkappa}} \, |u|^2\, dx \le 
\int_{\mathbb R^d} \left(\frac{\Phi(x/|x|)}{|x|^{2\varkappa}}\right)^* \, (u^*)^2\, dx\\
=
\frac{\|\Phi\|_{L^{d/2\varkappa}(\mathbb S^{d-1})}}{\left|\mathbb S^{d-1}\right|^{2\varkappa/d}}
\, \int_{\mathbb R^d} \frac{(u^*)^2}{|x|^{2\varkappa}} \, dx
\le C_\varkappa \, \frac{\|\Phi\|_{L^{d/2\varkappa}(\mathbb S^{d-1})}}{\left|\mathbb S^{d-1}\right|^{2\varkappa/d}} \,  
\int_{\mathbb R^d} |\nabla^\varkappa u^*(x)|^2\, dx. 
\end{multline*} 
Finally by using the P\'olya and Szeg\"o rearrangement inequality (see for example \cite{P}, \cite{LL}). 
$$
\|\nabla^{\varkappa} u^*\|_2Ê\le \|\nabla^{\varkappa}  u \|_2, \qquad 0 \le \varkappa \le 1,
$$
we complete the proof of Theorem \ref{FracLaplacians}.


\end{document}